\documentclass [12pt,a4paper]{amsart}
\usepackage{amsmath}
\usepackage{stmaryrd}
\usepackage{textcomp}

\usepackage{cases}
\usepackage{amscd}
\usepackage{epsfig}
\usepackage{graphicx}

\usepackage{soul}

\usepackage{color}

\usepackage{verbatim}

\usepackage{amssymb}
\usepackage{amsfonts}
\usepackage{amsbsy}

\usepackage{amsthm}
\usepackage{amstext}
\usepackage{amsopn}

\renewcommand{\Im}{\textrm{Im}}

\newcommand{\calH}{\mathcal{H}}

\newcommand{\calC}{\mathcal{C}}

\newcommand{\calF}{\mathcal{F}}

\newcommand{\R}{\mathbb{R}}

\newcommand{\mH}{\mathbb{H}}

\newcommand{\frakH}{\mathfrak{H}}

\newcommand{\na}{\nabla}

\newcommand{\bi}{{\bf i}}
\newcommand{\bj}{{\bf j}}
\newcommand{\bk}{{\bf k}}

\theoremstyle{plain}
\newtheorem{thm}{Theorem}[section]
\newtheorem{cor}[thm]{Corollary}
\newtheorem{lem}[thm]{Lemma}
\newtheorem{prop}[thm]{Proposition}
\theoremstyle{definition}
\newtheorem{defi}[thm]{Definition}
\newtheorem{rem}[thm]{Remark}

\newtheorem{eg}[]{\emph{Example}}
\numberwithin{equation}{section}

\usepackage{fancyhdr}
\begin{document}

\title
[sub-Riemannian geometry of the quaternionic Heisenberg group]
{On the sub-Riemannian geometry of the quaternionic Heisenberg group}
\author{Joonhyung Kim \& Ioannis D. Platis \& Li-Jie Sun}

\address{Department of Mathematics Education,
Chungnam National University,
 99 Daehak-ro, Yuseong-gu, Daejeon 34134,
 Korea.}
\email{calvary@cnu.ac.kr}

\address{Department of Mathematics,
University of Patras,
 Rion, Achaia, 26504,
 Greece.}
\email{idplatis@upatras.gr}
\address {Department of Applied Science, Yamaguchi University,
2-16-1 Tokiwadai, Ube 7558611,
Japan.
}
\email{ljsun@yamaguchi-u.ac.jp}

\keywords{Quaternionic Heisenberg group, Carnot-Carath\'eodory geodesics, horizontal mean curvature, horizontal minimal surfaces.}
\subjclass[2020]{Primary 53C17; Secondary 53C42, 53C26.}
\thanks{{\it Acknowledgement.} The first author was supported by the National Research Foundation of Korea (NRF) grant funded by the Korea government (RS-2024-00345250). The second author was funded by the Medicus programme, No. 83765.}

%
\begin{abstract}

Utilizing the framework of quaternionic contact geometry, we define a sequence of Riemannian metrics $\{g_L\}$ on the quaternionic Heisenberg group $\frakH_{\mathbb{H}}$ by rescaling the vertical directions. By analyzing the limit of this sequence, we characterize the Carnot-Carath\'eodory geodesics and provide the explicit description of the Carnot-Carath\'eodory distance and spheres in $\frakH_{\mathbb{H}}$. Furthermore, we derive a general formula for the horizontal mean curvature of hypersurfaces. 

\end{abstract}
\maketitle

\section{Introduction}

The complete non-compact symmetric spaces of rank one consist of the real, complex, and quaternionic hyperbolic spaces, alongside the Cayley hyperbolic plane. The boundary of each space plays a key role in the  geometry and analytic structure. While the boundary of real hyperbolic space carries a conformal structure, the boundary of complex hyperbolic space is endowed with a spherical CR structure. In contrast, the boundary of quaternionic hyperbolic space is characterized by a quaternionic contact (QC) structure.

Specifically, the boundaries of the complex and quaternionic hyperbolic spaces are identified with the one-point compactifications of the Heisenberg group and the quaternionic Heisenberg group, respectively. Both are Carnot groups and fundamental examples of 2-step nilpotent Lie groups, yet they differ fundamentally in their algebraic and geometric dimensions: the center of the standard Heisenberg group is 1-dimensional, while that of the quaternionic Heisenberg group $\frakH_\mH$ is 3-dimensional. While the geometry and analysis of the Heisenberg group have been extensively documented (see, e.g., \cite{CDPT, Fol, Mont}), the quaternionic Heisenberg group $\frakH_{\mathbb{H}}$ remains less explored. This gap is largely due to the algebraic complexity arising from the non-commutativity of the quaternions, which significantly complicates the study of its sub-Riemannian structure and QC structures.

In this paper, we focus on the 7-dimensional quaternionic Heisenberg group $\frakH_{\mathbb{H}}$, modeled on the manifold $\mathbb{H} \times \text{Im}(\mathbb{H})$. Our work aims to extend foundational results from the standard Heisenberg group to this quaternionic setting. 

First, we address the sub-Riemannian geometry of the group. A Carnot-Carath\'eodory (CC) metric $d_{cc}$ can be defined on any Carnot group, see \cite{BLU, Mit}. We explicitly describe the structure of horizontal curves joining arbitrary points in $\frakH_{\mathbb{H}}$ in Section \ref{sec:HC}. 

A central theme of this work is the Riemannian approximation of the CC-metric, see Section \ref{sec:app}. We fix a left-invariant metric $g_{cc}$ on the horizontal bundle and extend it to the entire tangent bundle. For a triple of positive scaling parameters $L=(L_1, L_2, L_3)$, we define a family of Riemannian metrics $g_L$ by rescaling the vertical directions. The CC-geodesics in $\mathfrak{H}_{\mathbb{H}}$ are then characterized as the limits of the $g_L$-geodesics as the parameters $L_a\, (a=1, 2, 3)$ tend to infinity (see Section \ref{subsec:CC-geo}). This follows from \cite{CDPT}.
In \cite{KP}, Kim and Parker established a comparison between the Korányi distance and the CC-distance by extending the results obtained for the standard Heisenberg group by Basmajian and Miner in \cite{BM}. In the present work, we provide an alternative comparison theorem in Section \ref{subsec:KC}, utilizing a distinct approach based on our Riemannian approximation framework.
We also clarify its structure and compute its horizontal mean curvature (see Section \ref{sec:hyp}), providing a foundational tool for addressing the isoperimetric problem in the quaternionic Heisenberg group. Furthermore, we identify two explicit examples of horizontal minimal surfaces, providing rare concrete cases in a non-commutative sub-Riemannian environment.

The paper is organized as follows. Section \ref{sec:pre} provides background on quaternionic contact and CR structures. In Section \ref{sec:HC}, we construct explicit horizontal paths. Section \ref{sec:app} develops the Riemannian approximations $g_L$, including computations of Ricci and scalar curvatures. Section \ref{sec:geo} analyzes $g_L$-geodesics and the CC-geodesics. Finally, Section \ref{sec:hyp} investigates the horizontal mean curvature of hypersurfaces and presents the new horizontal minimal surfaces.

\section{Preliminaries}\label{sec:pre}
\subsection{Quaternions}
We denote quaternions by
$a+b\bi+c\bj+d\bk,$
where $a, b, c, d\in \R;\, \bi,\, \bj,$ and $\bk$ are the basis vectors satisfying
$$\bi^2=\bj^2=\bk^2=\bi\bj\bk=-1.$$ 
The algebra of quaternions is denoted by $\mH$; it is associative and non-commutative. The imaginary quaternions are elements of the set $\Im\mH=\rm{Span}_\R\{\bi, \bj, \bk\}\simeq\R^3.$ Let $q=x_1+x_2\bi+x_3\bj+x_4\bk\in\mH$; then its conjugate $\bar{q}$, its imaginary part $\Im(q)$ and its modulus $|q|$ are respectively given by
\begin{eqnarray*}
&&\bar{q}=x_1-x_2\bi-x_3\bj-x_4\bk,\\
&&\Im q=x_2\bi+x_3\bj+x_4\bk,\\
&&|q|^2=h\bar{h}=x_1^2+x_2^2+x_3^2+x_4^2.
\end{eqnarray*}
The exponential of a non-zero imaginary quaternion is give by
\[
{\rm exp}\, q={\rm exp}\, (x_1+{\bf v})=e^{x_1}\left(\cos|{\bf v}|+\frac{{\bf v}}{|{\bf v}|}\sin{|{\bf v}|}\right).
\]
\subsection{Quaternionic CR and  contact structures} Suppose that $M$ is a $(4n+3)$-dimensional manifold.  Let $\eta=(\eta_1, \eta_2, \eta_3)$ be an $\R^3$-valued 1-form defined on $M$ with codimension $3$ distribution $\calH={\rm Ker}\,\eta={\rm Ker}\, \eta_1\cap{\rm Ker}\, \eta_2\cap{\rm Ker}\, \eta_3.$ 
Suppose that $J_a$, $a=1, 2, 3,$ are the complex structures on $\calH$ satisfying the unit quaternionic relations: $J_1^2=J_2^2=J_3^2=J_1J_2J_3=-{\rm id}$.

We consider a triple of linearly independent 1-forms $\eta=(\eta_1, \eta_2, \eta_3)$; associated to them are six 2-forms $\rho_\sigma$, $\sigma=(a,b,c)\in S_3$, defined by
\begin{equation}\label{eq-assoc-forms}
\rho_\sigma=d\eta_a+2\eta_b\wedge\eta_c.
\end{equation}
Here, $S_3$ is the 3-permutation group. We fix this notation throughout the paper. 
\begin{defi}\label{defi:qCR}
A triple of linearly independent 1-forms $\eta=(\eta_1, \eta_2, \eta_3)$ is called a {\it quaternionic CR (qCR) structure} if the associated 2-forms  given by (\ref{eq-assoc-forms}) satisfy the following:
\begin{enumerate}
\item[{(1)}] they are non-degenerate on the horizontal distribution $\calH$ and have the same 3-dimensional kernel;
\item [{(2)}] for $\sigma=(a,b,c)\in S_3$, the endomorphisms $J_a$ of the horizontal distribution $\calH$ defined by
$J_a=(\rho_c|_\calH)^{-1}\cdot\rho_b|_\calH$, $a=1,2,3$,
satisfy the unit quaternionic relations.
\end{enumerate}
\end{defi}
\begin{defi}\label{defi:QC}
Let $\sigma=(a,b,c)\in S_3$. The three 2-forms $d\eta_a$ of the associated forms $\rho_\sigma$ are called {\it fundamental} if there exists a positive definite symmetric tensor  $g$ defined on  $\calH$ satisfying
$$d\eta_a(X, Y)=g(J_a X, Y),\, a=1, 2, 3,$$
for any vector fields $X, Y$ in $\calH$. We then call $\{\eta, \{J_a\}_{a=1}^3, g\}$  a {\it quaternionic contact (QC) structure on $M$ with horizontal distribution $\calH$ equipped with an $CSp(n)Sp(1)$ structure.} 
\end{defi}
Here, note that the 1-form $\eta$ is given only up to the action of $SO(3)$ on $\R^3$ and to a conformal factor, thus one can get a $CSp(n)Sp(1)$-structure on $\calH.$

\subsection{The quaternionic Heisenberg group and its QC structure}\label{subsec:QC}
The quaternionic Heisenberg group $\frakH_{\mH}$ can be identified with the set $(\mH, \Im\mH)$ with the group law
\begin{equation}
(q', t')\ast(q,t)=(q'+q, t'+t+2\Im(\overline{q}\cdot q')),
\end{equation}
where $q, q'\in\mH,$ $t,\,t'\in\Im\mH$ and $\overline{q'}$ is the quaternionic conjugate. We have the following automorphisms of $\frakH_\mH:$
\begin{enumerate}
\item[{(1)}] Left translations $\tau_{(q', t')}:$
$$\tau_{(q', t')}(q, t)=(q', t')\ast(q, t);$$
\item[{(2)}] Rotations $U:$
$$U(q, t)=(qU, t), \quad {\rm for}\quad U\in{\rm Sp(1)},$$
where ${\rm Sp(1)}$ is the set of unit quaternions.
\item[{(3)}] ${\rm Sp(1)}$ acts as:
$$\sigma(q, t)=(\sigma q, \sigma t\sigma^{-1}),\quad {\rm for}\quad \sigma\in{\rm Sp(1)},$$
where $\sigma$ acts on the second factor as isomorphism with ${\rm SO(3)}.$
\item[{(4)}] The inversion $I:$
$$
I(q, t)=\left(-(|q|^2-t)^{-1}q, \frac{-t}{|q|^4+|t|^2}\right);
$$
\item[{(5)}]  Real dilations $D_\delta:$ 
$$D_\delta(q, t)=(\delta q, \delta^2 t), \quad {\rm where}\quad \delta>0.$$
\end{enumerate}
The left translations (type (1)), quaternionic Heisenberg rotations (types (2) and (3)) and real dilations (type (5)) generate the {\it quaternionic Heisenberg similarity group} ${\rm Sp}(1)\times {\rm SO(3)}\times \R_+\rtimes\frakH_\mH.$
The vector fields
\begin{eqnarray*}
\xi_1&&=\displaystyle{\frac{\partial}{\partial x_1}+2x_2\frac{\partial}{\partial t_1}+2x_3\frac{\partial}{\partial t_2}+2x_4\frac{\partial}{\partial t_3}},\\
\xi_2&&=\displaystyle{\frac{\partial}{\partial x_2}-2x_1\frac{\partial}{\partial t_1}-2x_4\frac{\partial}{\partial t_2}+2x_3\frac{\partial}{\partial t_3}},\\
\xi_3&&=\displaystyle{\frac{\partial}{\partial x_3}+2x_4\frac{\partial}{\partial t_1}-2x_1\frac{\partial}{\partial t_2}-2x_2\frac{\partial}{\partial t_3}},\\
\xi_4&&=\displaystyle{\frac{\partial}{\partial x_4}-2x_3\frac{\partial}{\partial t_1}+2x_2\frac{\partial}{\partial t_2}-2x_1\frac{\partial}{\partial t_3}},
\end{eqnarray*}
as well as the vector fields $T_i=\partial/\partial t_i$ are left-invariant and form a basis for the Lie algebra.
We have the following bracket relations:
\begin{eqnarray*}
&&
[\xi_1,\xi_2]=[\xi_3,\xi_4]=-4T_1,\\
&&
[\xi_1,\xi_3]=[\xi_4,\xi_2]=-4T_2,\\
&&
[\xi_1,\xi_4]=[\xi_2,\xi_3]=-4T_3.
\end{eqnarray*}
We now describe in brief the QC structure of $\frakH_{\mH}$: let $q=x_1+x_2 \bi+x_3 \bj+x_4 \bk$ and $t=t_1\bi+t_2\bj+t_3\bk.$ The standard $\R^3$-valued contact form of $\frakH_{\mH}$ is given by
\begin{equation}
\omega=dt-q\cdot d\bar{q}+dq\cdot \bar{q}=\theta_1\,\bi+\theta_2\,\bj+\theta_3\,\bk,
\end{equation}
where
\begin{equation}\label{contactform}
\begin{aligned}
\theta_1&&=dt_1-2x_2dx_1+2x_1dx_2+2x_3dx_4-2x_4dx_3,\\
\theta_2&&=dt_2-2x_3dx_1+2x_1dx_3+2x_4dx_2-2x_2dx_4,\\
\theta_3&&=dt_3-2x_4dx_1+2x_1dx_4+2x_2dx_3-2x_3dx_2.
\end{aligned}
\end{equation}
Note that $\{dx_1, dx_2, dx_3,dx_4, \theta_1, \theta_2, \theta_3\}$ is the dual basis of $\{\xi_1,\xi_2,\xi_3,\xi_4,T_1,T_2,T_3\}$.
Now, for $\sigma=(a,b,c)\in S_3$,
$$\ker\theta_a={\rm Span}_{\R}\left\{\xi_1,\xi_2,\xi_3,\xi_4, T_b, T_c\right\}. 
$$
Let $\calH=\ker\, \theta_1\cap{\ker}\, \theta_2\cap{\ker}\, \theta_3,$ which is a codimension $3$ distribution spanned by $\xi_1, \xi_2, \xi_3, \xi_4.$ 
A metric $g_{cc}$ is defined on the horizontal distribution $\calH$ by
$g_{cc}(\xi_i, \xi_j)=\delta_{ij}$, $j=1,\dots, 4.$
This is the {\it Carnot-Carath\'eodory metric} defined on $\calH$. Explicitly, 
$$g_{cc}=\sum_{i=1}^4dx_i^2.$$
\begin{rem}
Automorphisms (1)--(4) above preserve the metric $g_{cc}$ and they are called {\it quaternionic Heisenberg isometries}. Dilations given by $(5)$ scale $g_{cc}$ up to the dilation factor. 
\end{rem}
The three almost complex structures
 $\{J_1, J_2, J_3\}$ on $\calH$ 
are given by 
$$d\theta_a(X, Y)=d\theta_b(J_c X,Y),$$
for all $\sigma=(a,b,c)\in S_3$ and for all $X, Y\in\calH.$ 
Direct computations yield that
\begin{eqnarray*}
&&J_1\xi_1=\xi_2, J_1\xi_2=-\xi_1,J_1\xi_3=\xi_4, J_1\xi_4=-\xi_3,\\
&&J_2\xi_1=\xi_3, J_2\xi_3=-\xi_1,J_2\xi_4=\xi_2, J_2\xi_2=-\xi_4,\\
&&J_3\xi_1=\xi_4, J_3\xi_4=-\xi_1,J_3\xi_2=\xi_3, J_3\xi_3=-\xi_2.
\end{eqnarray*}
One verifies directly that
$$
J_1^2=J_2^2=J_3^2=-{\rm Id},\quad J_aJ_b=-J_bJ_a=J_c
$$
and that for any $X, Y\in\calH,$ we have:
$$d\theta_a(X, Y)=2\,g_{cc}(J_a X, Y)=d\theta_a(J_{a}X, J_{a}Y),\quad a=1,2,3.
$$
Let $T_a={\partial}/{\partial t_a},$ then $T_a$ is the Reeb vector field of $\theta_a,$ $a=1,2,3.$
Now, we see that $(\frakH_{\mH};\omega, \{J_a\}^3_{a=1}, g_{cc})$
is a QC-manifold with distribution $\calH.$ 
By Definition \ref{defi:qCR}, $\{\omega, J_1,\, J_2,\, J_3\}$ is a qCR structure on $\frakH_{\mH}$. Further details on the quaternionic Heisenberg group and its fundamental structure can be found in \cite{KP, WS}.

\section{Horizontal curves in $\frakH_\mH$} \label{sec:HC}
Suppose that a curve $\gamma (\lambda):[0, 1]\rightarrow \frakH_{\mathbb{H}}$ is defined by 
$\gamma(\lambda)=(\alpha(\lambda), \beta(\lambda)).$
The curve $\gamma(\lambda)$ is called {\it horizontal} curve if and only if  $\omega(\dot\gamma(\lambda))=0.$
\begin{lem}\label{lem:hor}
An horizontal curve remains horizontal under left translations as well as under dilations. 
\end{lem}
One can check that $\omega(\dot\tau(\gamma))=0$ and $\omega(\dot D_\delta(\gamma))=0$ by a straightforward computation.

The existence of horizontal paths between any two points in Carnot groups (like the quaternionic Heisenberg group) is a consequence of Chow-Rashevskii's theorem, see \cite{Chow, Ras}. However, explicit constructions in the quaternionic Heisenberg group are not commonly available in the literature. We include the construction here for clarity, completeness and potential use in applications. We start with the following:
\begin{lem}\label{lem:ver}
There are horizontal curves joining the origin with the point $(0, t)\in(\mH, {\emph\Im}\mH).$
\end{lem}
\begin{proof}
Let $ k_i \in \R_*$,  $i = 1, 2, 3, 4$. 
We first consider the curve $\gamma_1$ given by $\gamma_1(\lambda)=(k_1\lambda,0)$, $\lambda\in[0,1]$. We have that $\dot\gamma_1(\lambda)=k_1\xi_1(\gamma_1(\lambda))$ and this curve joins the origin with the point $p_1=(k_1,0)$. We now consider the horizontal curve $\gamma_2$ given by
$$
\gamma_2(\lambda) = \left(k_1 + k_2 \lambda\, \bi,\, -2k_1 k_2 \lambda \,\bi \right), \quad \lambda\in[0,1].
$$
We have that $\dot\gamma_2(\lambda)=k_2\xi_2(\gamma_1(\lambda))$ and $\gamma_2$ joins $p_1$ and the point $ p_2 = \left(k_1 + k_2 \bi,\, -2k_1 k_2 \bi\right) $. 
Similarly, starting from $p_2$, we move in the direction of $k_3\xi_3$ along the curve $\gamma_3$ given by 
$$
\gamma_3(\lambda) = \left(k_1 + k_2 \bi + k_3 \lambda\, \bj,\, -2k_1 k_2 \bi - 2(k_1 k_3 \bj + k_2 k_3 \bk)\lambda \right), \quad \lambda\in[0,1],
$$
which joins $p_2$ and $ p_3 = \left(k_1 + k_2 \bi + k_3 \bj,\ -2(k_1 k_2 \bi + k_1 k_3 \bj + k_2 k_3 \bk)\right) $. Continuing in this fashion, from $p_3$, we move in the direction of the vector field $k_4 \xi_4$ along the curve   $\gamma_4$ given for each $\lambda\in[0,1]$ by
$$
\gamma_4(\lambda) = \left(k_1 + k_2 \bi + k_3 \bj + k_4 \lambda\, \bk,\ -2\left((k_1 k_2 + k_3 k_4 \lambda) \bi + (k_1 k_3 - k_2 k_4 \lambda) \bj + (k_2 k_3 + k_1 k_4 \lambda) \bk\right) \right);
$$ 
this curve joins $p_3$ and the point  
$$
p_4 = \left(k_1 + k_2 \bi + k_3 \bj + k_4 \bk,\ -2\left((k_1 k_2 + k_3 k_4) \bi + (k_1 k_3 - k_2 k_4) \bj + (k_2 k_3 + k_1 k_4) \bk\right) \right).
$$
In a similar manner:
\begin{itemize}
\item
From $p_4$, we travel in the direction of $ -k_1 \xi_1 $, reaching the point  
$$
p_5 = \left(k_2 \bi + k_3 \bj + k_4 \bk,\ -2\left((2k_1 k_2 + k_3 k_4) \bi + (2k_1 k_3 - k_2 k_4) \bj + (k_2 k_3 + 2k_1 k_4) \bk\right) \right);
$$
\item from $p_5$, we follow $-k_2 \xi_2 $, arriving at  
$$
p_6 = \left(k_3 \bj + k_4 \bk,\ -2\left((2k_1 k_2 + k_3 k_4) \bi + (2k_1 k_3 - 2k_2 k_4) \bj + (2k_2 k_3 + 2k_1 k_4) \bk\right) \right);
$$
\item
from $p_6$ we proceed along $ -k_3 \xi_3 $, reaching  
$$
p_7 = \left(k_4 \bk,\ -4\left((k_1 k_2 + k_3 k_4) \bi + (k_1 k_3 - k_2 k_4) \bj + (k_2 k_3 + k_1 k_4) \bk\right) \right)
$$
and finally,
\item from $p_7$, we move in the direction of $ -k_4 \xi_4 $, reaching the point  
$$
p_8 = \left(0,\ -4\left((k_1 k_2 + k_3 k_4) \bi + (k_1 k_3 - k_2 k_4) \bj + (k_2 k_3 + k_1 k_4) \bk\right) \right).
$$
\end{itemize}
Thus, the horizontal path comprising the union of all horizontal paths above joins the origin and the point whose vertical component depends linearly on the quantities $k_1 k_2 + k_3 k_4 $, $ k_1 k_3 - k_2 k_4 $, and $ k_2 k_3 + k_1 k_4 $. By choosing suitable values 
for the coefficients $k_i$, we claim that there always exists a family of horizontal curves connecting the origin to the point $(0, t)$. Actually, if $t=(-4t_1, -4t_2, -4t_3),$ we must solve
\begin{eqnarray*}
k_1 k_2 + k_3 k_4&=&t_1,\\
k_1 k_3 - k_2 k_4&=&t_2,\\
k_2 k_3 + k_1 k_4 &=&t_3.
\end{eqnarray*}
Let $t_2 + i t_3 = R e^{i\phi}.$ Suppose that $z_1 =k_1 + k_2 i,$ and $z_2 = k_3 + k_4i$. Using the second and third equations above, we note that
$$z_1 z_2 = (k_1k_3 - k_2k_4) + i(k_1k_4 + k_2k_3) = t_2 + i t_3.$$
When $R\neq0,$ by setting $z_1 = r_1 e^{i\phi_1}$ and $z_2 = r_2 e^{i\phi_2}$, we derive that $r_1 r_2 = R,$ and $\phi_2 = \phi - \phi_1.$ 
The equation $k_1k_2 + k_3k_4 = t_1$ can be rewritten as 
$$\frac{1}{2} r_1^2 \sin(2\phi_1) + \frac{1}{2} r_2^2 \sin(2\phi_2) = t_1.$$
Substituting the expressions for $r_2$ and $\phi_2$:
$$r_1^2 \sin(2\phi_1) + \frac{R^2}{r_1^2} \sin(2\phi- 2\phi_1) = 2t_1.$$
By letting $X = r_1^2 > 0,$ we get that
$$(\sin 2\phi_1) X^2 - (2t_1) X + R^2 \sin(2\phi - 2\phi_1) = 0.$$
A solution exists if there exists a $\phi_1 \in [0, 2\pi)$ such that the quadratic has at least one positive real root. By the intermediate value theorem, we can select $\phi_1$ such that the coefficients $\sin(2\phi)$ and $R^2 \sin(2\phi - 2\phi_1)$ have opposite signs. In such a case, the product of the roots is negative, ensuring the existence of one positive real root for $X$.
The case where $R=0$ follows trivially by setting $k_1=k_2=0$, yielding $k_3k_4=t_1$, which is solvable for all $t_1$.
 This completes the proof.
\end{proof}
\begin{prop}\label{thm:hor}
For any two points in $\frakH_{\mathbb{H}},$ there always exists a horizontal path joining them.
\end{prop}
\begin{proof} Given two points $P=(q, t)$ and  $P'=(q', t'),$
we define the projection $\pi: \frakH_{\mathbb{H}}\to \mathbb{H}$ by $\pi(q, t)=q.$
First, consider an absolutely continuous curve $\alpha(\lambda):[0, 1]\rightarrow\mathbb{H}$ satisfying $\alpha(0)=\pi(P)$ and $\alpha(1)=\pi(P').$ Then we lift the curve $\alpha(\lambda)$ to 
a horizontal curve $\gamma(\lambda):[0, 1]\rightarrow\frakH_\mathbb{H}$ such that 
$$\pi\circ\gamma(\lambda)=\alpha(\lambda ), \quad\gamma(0)=P.$$
We claim that such a curve does exist. Indeed, we only need to define $\beta(\lambda):[0, 1]\to\Im\mH$ as follows:
\begin{eqnarray*}
&&\beta_1(\lambda)=t_1+2\int_0^{\lambda }(\alpha_2\alpha_1'-2\alpha_1\alpha_2'-2\alpha_3\alpha_4'+2\alpha_4\alpha_3')(s)ds\\
&&\beta_2(\lambda )=t_2+2\int_0^{\lambda }(\alpha_3\alpha_1'-2\alpha_1\alpha_3'-2\alpha_4\alpha_2'+2\alpha_2\alpha_4')(s)ds,\\
&&\beta_3(\lambda )=t_3+2\int_0^{\lambda }(\alpha_4\alpha_1'-2\alpha_1\alpha_4'-2\alpha_2\alpha_3'+2\alpha_3\alpha_2')(s)ds,
\end{eqnarray*}
where $(t_1, t_2, t_3)$ are the vertical components of the initial point $P.$
We have thus constructed a horizontal curve \(\gamma(\lambda) \) joining the point \( P \) to a point \( R \) with coordinates  
$
(\alpha(1), \beta(1)) = (q',\ \beta(1)).
$
Assume that  
\[
(-q', -t') \ast (q',\ \beta(1)) = (0, a) \in (\mH, \Im \mH).
\]  
By Lemma~\ref{lem:ver}, there exists a horizontal curve joining the point \( (0, a) \) and the origin.  
Then, applying Lemma~\ref{lem:hor}, we conclude that there exists a  horizontal curve connecting the point \( R \) to the point \( P' \).  
This completes the proof.
\end{proof}
The {\it horizontal length} of a curve  $ \gamma$ is defined by  
$$
l_{cc}(\gamma) = \int_0^1 \left| (\pi \circ \gamma)'(\lambda) \right| \, d\lambda,
$$ 
where $\pi$ denotes the projection onto the horizontal component.

The {\it Carnot-Carathéodory distance} between two points $P$ and $P'$ in the quaternionic Heisenberg group is then given by  
$$
d_{cc}(P, P') = \inf_{\gamma\in \Gamma}\{l(\gamma)\},
$$
where $\Gamma$ is the set comprising all horizontal curves joining $P$ and $P'.$

\section{Riemannian approximation}\label{sec:app}
Given $L_a>0$, $a=1,2,3,$  we define a Riemannian metric $g_L$ in for any vector fields $X,Y,Z$ of $\frakH_{\mH}$ such that the frame
$$
\{\xi_1,\xi_2,\xi_3,\xi_4, (1/L_1)T_1, (1/L_2)T_2,(1/L_3)T_3\}
$$
is orthonormal. The corresponding co-frame is
$$
\{dx_1,dx_2,dx_3,dx_4, L_1\theta_1,L_2\theta_2,L_3\theta_3\}.
$$
Let $T_a'=({1}/{L_a})T_a$ and $\theta_a'=L_a\theta_a,$ $a=1, 2, 3.$ 
 Then $g_L$ can be written as
$$g_L=\sum_{i=1}^4 dx_i^2+{\sum_{a=1}^{3}\theta_a'^2}.$$
Observe that the restriction of $g_L$ in $\calH$ is $g_{cc}$.

Recall Koszul's formula for the Riemannian connection $\nabla$ of $g_L$:
$$
g_L(\na_XY, Z)=\frac{1}{2}\left(g_L([X, Y], Z)-g_L([X,Z], Y)-g_L([Y, Z], X)\right),
$$
for any vector fields $X,Y,Z$ of $\frakH_{\mH}$.
We calculate straightforwardly:
\begin{eqnarray*}
&&  \na_{\xi_1} \xi_2 =\na_{\xi_3} \xi_4 =- \na_{\xi_2} \xi_1=- \na_{\xi_4} \xi_3=-2L_1T_1',\\
&& \na_{\xi_1} \xi_3 =\na_{\xi_4} \xi_2=- \na_{\xi_3} \xi_1=-\na_{\xi_2} \xi_4 =-2L_2T_2',\\
&& \na_{\xi_1} \xi_4 = \na_{\xi_2} \xi_3 =- \na_{\xi_4} \xi_1=- \na_{\xi_3} \xi_2=-2L_3T_3',\\
&& \na_{T_1'} \xi_1=2L_1\xi_2,\quad \na_{T_1'} \xi_2=-2L_1\xi_1,\quad  \na_{T_1'} \xi_3=2L_1\xi_4,\quad  \na_{T_1'} \xi_4=-2L_1\xi_3,\\
&& \na_{T_2'} \xi_1=2L_2\xi_3,\quad \na_{T_2'} \xi_2=-2L_2\xi_4,\quad  \na_{T_2'} \xi_3=-2L_2\xi_1,\quad  \na_{T_2'} \xi_4=2L_2\xi_2,\\
&& \na_{T_3'} \xi_1=2L_3\xi_4,\quad \na_{T_3'} \xi_2=2L_3\xi_3,\quad  \na_{T_3'} \xi_3=-2L_3\xi_2,\quad  \na_{T_3'} \xi_4=-2L_3\xi_1,
\end{eqnarray*}
and $ \na_{\xi_i} \xi_i =0,\, \na_{T_a'} T_a '=0,\, \na_{\xi_i} T_a'= \na_{T_a'} \xi_i ,$ for any $ i=1, 2, 3, 4,\, a=1, 2, 3.$

Recall also that the curvature $R$  is defined by
\begin{equation}
R(X, Y) Z=\nabla_Y\nabla_X Z-\nabla_X\nabla_Y Z+\nabla_{[X,Y]} Z,
\end{equation}
for any vector fields $X,Y,Z$ of $\frakH_{\mH}$.
Again, direct computations yield that 
\begin{eqnarray*}
&& R(\xi_1, \xi_j)\xi_1=-12L_{j-1}^2\xi_j \quad (j=2, 3, 4),\\
&&R(\xi_3, \xi_4)\xi_3=-12L_1^2\xi_4,\,\, R(\xi_2, \xi_4)\xi_2=-12L_2^2\xi_4,\,\,  R(\xi_2, \xi_3)\xi_2=-12L_3^2\xi_3,  \\
&& R(\xi_i, T_a')\xi_i=4L_a^2 T_a' \quad (i=1, 2, 3, 4, \,{\rm and}\,\, a=1, 2, 3).
\end{eqnarray*}
We have the following:
\begin{prop}
The sectional curvatures of the planes spanned by pairs of vector fields in $\{\xi_1,\xi_2,\xi_3,\xi_4, T_1', T_2', T_3'\}$, are given respectively by
\begin{eqnarray*}
&& K(\xi_1, \xi_2)=K(\xi_3, \xi_4)=-12L_1^2,\\
&& K(\xi_1, \xi_3)=K(\xi_2, \xi_4)=-12L_2^2,\\
&& K(\xi_1, \xi_4)=K(\xi_2, \xi_3)=-12L_3^2,\\
&& K(\xi_i, T_a')=4L_a^2.
\end{eqnarray*}
\end{prop}
\begin{proof}
Applying the formula $K(U, V)=g_L(R(U, V)U, V)$ for sectional curvature of planes spanned by unit vectors $U$ and $V$, the proof is straightforward. 
\end{proof}
\begin{cor}
The Ricci curvatures in each direction are, respectively,
\begin{alignat*}{2}
&\mathrm{Ric}(\xi_i) \; =\; -\frac{4}{3}(L_1^2 + L_2^2 + L_3^2), \quad && i = 1, 2, 3, 4, \\
&\mathrm{Ric}(T_a') \;   =\; 2L_a^2,              \quad && a = 1, 2, 3.
\end{alignat*}
Moreover the scalar curvature is
$$K
=-\frac{10}{21}(L_1^2+L_2^2+L_3^2).$$
\end{cor}

\section{Geodesics}\label{sec:geo}
\subsection{Geodesics of the metric $g_L$}
There exists an extensive study of geodesics and their properties in the Heisenberg group; for further reference, see \cite{BM, LM, Mont, Mon}. 
In this section we explicitly calculate the $g_L$-geodesics. In the next section \ref{subsec:CC-geo} the Carnot-Carathéodory geodesics will appear by letting $L_a\to \infty\, (a=1, 2, 3).$
Let $p$ and $q$ be two arbitrary points and let also $\gamma(\lambda)=(\alpha(\lambda), \beta(\lambda)): [0, 1]\to\frakH_\mH$ be a curve joining them. Then the $g_L$-length of $\gamma$ is
\[
S_{\gamma}=\int_{0}^1\sqrt{\sum_{i=1}^4 (\dot \alpha_i(\lambda))^2+\sum_{a=1}^3L_a^2
|\theta_a(\dot\gamma(\lambda))|_{\gamma(\lambda)}|^2} d\lambda.
\]
We aim to find the minimizing curve among all the curves joining $p$ and $q.$ We focus on the energy functional 
$$E(\gamma)=\sum_{i=1}^4 (\dot\alpha_i(\lambda))^2+\sum_{a=1}^3L_a^2\left|\theta_a(\dot\gamma(\lambda))|_{\gamma(\lambda)}\right|^2$$
and we shall make use of the Euler--Lagrange equations to find the length minimizing curve joining $p$ and $q$.

First, we claim that 
\begin{equation}\label{energy}
\frac{d}{d\lambda}\left(\theta_a(\dot\gamma(\lambda))|_{\gamma(\lambda)}\right)=0,\quad a=1, 2, 3.
\end{equation} 
In fact, along the curve we compute 
$$
\frac{\partial E}{\partial {\dot\beta_{a}}}=2L_a^2\theta_a,\quad \frac{\partial E}{\partial \beta_{a}}=0.
$$
The equality  $\frac{d}{d\lambda}E_ {\dot\beta_{a}}=\frac{\partial E}{\partial \beta_{a}}$ implies that $2L_a^2\dot\theta_a=0,$ which proves (\ref{energy}).
Secondly, we also have:
\begin{align*}
\frac{\partial E}{\partial \dot\alpha_1}&=2\dot\alpha_1-4L_1^2\theta_1\alpha_2-4L_2^2\theta_2\alpha_3-4L_3^2\theta_3\alpha_4,\\
\frac{\partial E}{\partial \dot\alpha_2}&=2\dot\alpha_2+4L_1^2\theta_1\alpha_1+4L_2^2\theta_2\alpha_4-4L_3^2\theta_3\alpha_3,\\
\frac{\partial E}{\partial \dot\alpha_3}&=2\dot\alpha_3-4L_1^2\theta_1\alpha_4+4L_2^2\theta_2\alpha_1+4L_3^2\theta_3\alpha_2,\\
\frac{\partial E}{\partial \dot\alpha_4}&=2\dot\alpha_4+4L_1^2\theta_1\alpha_3-4L_2^2\theta_2\alpha_2+4L_3^2\theta_3\alpha_1.
\end{align*}
On the other hand, we have:
\begin{align*}
\frac{\partial E}{\partial {\alpha_1}}&=4L_1^2\theta_1\dot\alpha_2+4L_2^2\theta_2
\dot\alpha_3+4L_3^2\theta_3\dot\alpha_4,\\
\frac{\partial E}{\partial {\alpha_2}}&=-4L_1^2\theta_1\dot
\alpha_1-4L_2^2\theta_2\dot\alpha_4+
4L_3^2\theta_3\dot\alpha_3,\\
\frac{\partial E}{\partial {\alpha_3}}&=4L_1^2\theta_1\dot\alpha_4-
4L_2^2\theta_2\dot\alpha_1-4L_3^2\theta_3
\dot\alpha_2,\\
\frac{\partial E}{\partial {\alpha_4}}&=-4L_1^2\theta_1\dot\alpha_3
+4L_2^2\theta_2\dot\alpha_2-4L_3^2\theta_3
\dot\alpha_1.
\end{align*}
Due to (\ref{energy}) we may suppose that 
\begin{equation}\label{equ:the}
\theta_a(\dot\gamma(\lambda))|_{\gamma(\lambda)}=\frac{C_L^a}{4L_a^2},
\end{equation}
where $C_L^a$ are arbitrary real constants. From the Euler-Lagrange equations $\frac{d}{d\lambda}E_{\dot\alpha_i}=\frac{\partial E}{\partial \alpha_i}$ as well as from equations (\ref{equ:the}) it now follows that
\begin{align}\label{equ:ddgam}
\begin{split}
\ddot\alpha_1&=C_L^1\dot\alpha_2+C_L^2\dot
\alpha_3+C_L^3\dot\alpha_4,\\
\ddot\alpha_2&=-C_L^1\dot\alpha_1-C_L^2\dot\alpha_4+C_L^3\dot\alpha_3,\\
\ddot\alpha_3&=C_L^1\dot\alpha_4-C_L^2\dot\alpha_1-C_L^3\dot\alpha_2,\\
\ddot\alpha_4&=-C_L^1\dot\alpha_3+C_L^2\dot\alpha_2-C_L^3\dot\alpha_1.
\end{split}
\end{align} 
Integrating those equations, we obtain:
\begin{align}\label{equ:dgam}
\begin{split}
\dot\alpha_1&=C_L^1\alpha_2+C_L^2\alpha_3+C_L^3\alpha_4+c_1,\\
\dot\alpha_2&=-C_L^1\alpha_1-C_L^2\alpha_4+C_L^3\alpha_3+c_2,\\
\dot\alpha_3&=C_L^1\alpha_4-C_L^2\alpha_1-C_L^3\alpha_2+c_3,\\
\dot\alpha_4&=-C_L^1\alpha_3+C_L^2\alpha_2-C_L^3\alpha_1+c_4,
\end{split}
\end{align}
where $c_i$ are arbitrary constants,  $i=1, 2, 3, 4.$ The $g_L$-geodesics are the solutions of the system given by 
(\ref{equ:the}) and (\ref{equ:dgam}).
Let $C_L=-C_L^1\bi-C_L^2\bj-C_L^3\bk.$ 
Since the Riemannian approximants are left-invariant metrics, it is enough to give the explicit solutions of $g_L$-geodesics $\gamma(\lambda)$ such that $\gamma(0)=O$ and $\gamma(1)=(q,t).$ We distinguish two cases.

\smallskip
(I) If $C_L=0,$
we have
\begin{equation*}
\begin{cases}
\alpha(\lambda) =q\,\lambda,\\
\beta(\lambda)=0,
%
\end{cases}
\end{equation*}
for some $q\in\mH.$ 

\smallskip
(II) If $C_L\neq0,$ we may write (\ref{equ:ddgam}) as
\begin{equation*}
\ddot\alpha(\lambda)=C_L\dot\alpha(\lambda),
\end{equation*}
which is a quaternionic-valued ordinary differential equation. The general solution for $\dot\alpha(\lambda)$ is given by
$$\dot\alpha(\lambda)={\rm exp}(C_L\lambda)C,$$
where $C=c_1+c_2\bi+c_3\bj+c_4\bk=\dot\alpha(0)$ is a constant quaternion (cf. \cite{CM}).
Since $C_L$ is a pure quaternion, it satisfies $-C_L^2=|C_L|^2.$ Integrating the expression for $\dot\alpha(\lambda)$, we obtain
\begin{equation*}\label{equ:gam}
\alpha(\lambda)=-\frac{1}{C_L}(1-{\rm exp}({C_L \lambda}))C.
\end{equation*}
The squared modulus $|\alpha(\lambda)|$ satisfies
$$
|\alpha(\lambda)|^2=\sum_{i=1}^4\alpha_i^2=\left|\frac{C}{C_L}\right|^2(2-2\cos(|C_L|\lambda)).
$$
Substituting \eqref{equ:dgam} into \eqref{equ:the}, we obtain the following system for the vertical derivatives:
\begin{align}
\dot\beta_{1}(\lambda) &= \frac{C_L^1}{4L_1^2} + 2C_L^1 \sum_{i=1}^4 \alpha_i^2 + 2(c_1\alpha_2 - c_2\alpha_1 + c_3\alpha_4 - c_4\alpha_3), \nonumber \\
\dot\beta_{2}(\lambda) &= \frac{C_L^2}{4L_2^2} + 2C_L^2 \sum_{i=1}^4 \alpha_i^2 + 2(c_1\alpha_3 - c_2\alpha_4 - c_3\alpha_1 + c_4\alpha_2), \label{equ:betadot} \\
\dot\beta_{3}(\lambda) &= \frac{C_L^3}{4L_3^2} + 2C_L^3 \sum_{i=1}^4 \alpha_i^2 + 2(c_1\alpha_4 + c_2\alpha_3 - c_3\alpha_2 - c_4\alpha_1). \nonumber
\end{align}
Integrating these equations with respect to $\lambda$, we obtain special solutions for the vertical components $\beta(\lambda) = (\beta_1(\lambda), \beta_2(\lambda), \beta_3(\lambda)),$ given by
\begin{equation}
\beta_{a}(\lambda) = \left( \frac{\lambda}{4L_a^2} + \frac{4|C|^2}{|C_L|^3} f(|C_L|\lambda) \right) C_L^a, \quad a = 1, 2, 3,
\end{equation}
where $f(\lambda) = \lambda - \sin \lambda.$

In the symmetric case where $L_1=L_2=L_3=L$, the Riemannian geodesics admit an explicit representation, as summarized in the following proposition.

\begin{prop}
The $g_L$-geodesic $\gamma(\lambda) = (\alpha(\lambda), \beta(\lambda))$ joining the origin and the point $(q,t)$ satisfies the following equations:
 \begin{equation}\label{equ:geo}
\begin{cases}
\alpha(\lambda) = -\frac{1}{{C}_L} \left( 1 - e^{C_L \lambda} \right) C, \\
\beta(\lambda) = -\left( \frac{\lambda}{4L^2} + \frac{4|C|^2}{|C_L|^3} f(|C_L|\lambda) \right) C_L.
\end{cases}
\end{equation}
Here, $C\in\mH$ and $C_L$ is a pure quaternion.
The behavior of these geodesics depends on the initial conditions as follows:
\begin{enumerate}
\item[(i)] {\it Horizontal Case:} If $t=0$, the geodesic is a horizontal straight line of the form $\gamma(\lambda) = (q\lambda, 0)$. The geodesic length is given by $S_\gamma = |q|$.

\item[(ii)] {\it Vertical Case:} If $q=0$, the geodesic is a vertical segment of the form $\gamma(\lambda) = (0, t\lambda)$. The geodesic length is $S_\gamma = L|t|$.
 
\item[(iii)] {\it General Case:} If $q \neq 0$ and $t \neq 0$, the geodesic is given by \eqref{equ:geo}, where the parameter $C_L$ is determined by the relation:
$$
\frac{t + \frac{C_L}{4L^2}}{|q|^2} = \frac{-2C_L f(|C_L|)}{|C_L| f'(|C_L|)}.
$$
Furthermore, the constant $C$ is uniquely determined by $C_L$. Specifically, by setting $\lambda=1$, we have
$$
|C| = \frac{|q||C_L|}{2\sin(|C_L|/2)}.
$$
In this case, the length of the geodesic arc is $\frac{1}{4L}\sqrt{16L^2|C|^2+|C_L|^2}.$
\end{enumerate}
\end{prop}

\subsection{Carnot-Carathéodory geodesics}\label{subsec:CC-geo}
By taking the limit $L\to\infty$ in  (\ref{equ:geo}), we obtain the unit speed Carnot-Carathéodory geodesics $\gamma(\lambda)=\alpha(\lambda),\,\beta(\lambda))$   as follows:
 \begin{equation}\label{equ:cc-geo}
\begin{cases}
\alpha(\lambda) =-\frac{1}{A}(1-{\rm exp}({A \lambda}))B,\\
\beta(\lambda)=-\frac{4}{|A|^3}f(|A|\lambda)A,
\end{cases}
\end{equation}
where $A\in\Im\mH,$ $B\in\mH$ with $|B|=1.$ 
Here, the parameter $\lambda$ is exactly the Carnot-Carathéodory distance, $d_{cc} (O, \gamma(\lambda)),$ from the origin $O$ to the point $\gamma(\lambda)=(\alpha(\lambda),\beta(\lambda)).$ 

Furthermore, the CC-sphere $S_R$ of radius $R$ centered at the origin $O$ is defined as
$$S_R=\{(q, t)\in\frakH_\mH: d_{cc}(O, (q, t))=R\}.$$ 
\begin{lem}\label{lem:CC-sphere}
The CC-sphere $S_R$ is the set of points given by 
 \begin{equation}\label{equ:sphere}
 \left(-\frac{1}{A}(1-{\rm exp}(AR))B,  -\frac{4}{|A|^3}f(|A|R)A\right),
\end{equation}
where $A\in\Im\,\mH,$ $B\in\mH$ with $|B|=1$ and $f$ is the real  function with formula $f(\lambda)=\lambda-\sin\lambda$.
\end{lem}
\subsection{Kor\'anyi metric vs. CC-metric}\label{subsec:KC}
The {\it Kor\'{a}nyi metric} $d_K$ is another fundamental metric on the quaternionic Heisenberg group $\frakH_\mH.$ It is defined by the left-invariant Kor\'{a}nyi gauge $\|\cdot\|_K$: 
$$
d_K((q, t), (q', t'))= \|(q, t)^{-1}\ast(q', t')\|_K,
$$
where the Kor\'{a}nyi gauge  $\|(q, t)\|_K=||q|^2+t|^{1/2}.$ 
 A direct calculation shows that the Kor\'{a}nyi distance from the origin $O$ to the points on the CC-sphere $S_R$ defined by (\ref{equ:sphere}) is
\begin{equation*}
d_{K}(O, (q, t))=\|(q, t)\|_K=\frac{\left(4(1-\cos(|A|R))^2+16f^2(|A|R)\right)^{1/4}}{|A|}.
\end{equation*}
The following \ref{thm-comp} comprises the relationship between CC and Kor\'{a}nyi distances. This is comparable to Theorem 2.2 in \cite{KP}, where the authors studied related problems by extending the results of \cite{BM}. Nevertheless, the framework considered here differs from that of \cite{KP}.
\begin{thm}\label{thm-comp}
Consider the origin $O$ and the point $(q, t)\in\frakH_\mH.$ 
\begin{itemize}
\item[(i)] If $t\neq0,$ then 
\[
d_{cc}(O, (q, t))=\left(\frac{x_0^4}{4(1-\cos x_0)^2+16(x_0-\sin x_0)^2 } \right)^{\frac{1}{4}}d_{K}(O, (q, t)),
\]
where $x_0\in [0, 2\pi)$ is the unique solution to the equation 
\[
\frac{1-\cos x_0}{2(x_0-\sin x_0)}=\frac{|q|^2}{|t|}.
\]
    \item[(ii)] If $t = 0$, then 
    \[
    d_{cc}(O, (q, t)) = |q| = d_{K}(O, (q, t)).
    \]
\end{itemize} 
\end{thm}
Since any two points in $\frakH_\mH$ can be mapped, via a quaternionic Heisenberg similarity, to the origin and a point $(q, t),$ we obtain the following conclusion.

\begin{cor}
The Carnot-Carathéodory metric and the Kor\'{a}nyi metric on the quaternionic Heisenberg group $\frakH_\mH$ are quasi-isometric.
\end{cor}

\section{Horizontal Mean Curvature of a Hypersurface}\label{sec:hyp}
In this section, we extend the study of the horizontal geometry of hypersurfaces in the Heisenberg group, as presented in \cite{CDPT}, to the setting of the quaternionic Heisenberg group $\frakH_\mH$.
We begin by defining horizontal mean curvature for regular hypersurfaces of the quaternionic Heisenberg group. After proving a general statement about a class of rotationally symmetric hypersurfaces, we compute the horizontal mean curvature for three fundamental classes of spheres in $\frakH_\mH$, that is, the Euclidean sphere,  the Kor\'anyi sphere, and  the Carnot-Carathéodory sphere.

Let $x=(q, t)\in \frakH_\mH,$ where $q=(x_1, x_2, x_3, x_4)\in \mH$ and $t=(t_1, t_2, t_3)\in \Im\mH.$
We consider a $C^2$ regular hypersurface $S$ defined as 
$$
S=\{x\in\frakH_\mH: u(x)=0\}.
$$
The geometry of $S$ is governed by its horizontal distribution. The {\it horizontal gradient of $u$} is given by 
$$
\nabla_0 u =\sum_{i=1}^4 (\xi_i u)\,\xi_i,
$$
and the {\it non-characteristic locus} of $S$ is the set of points where  the horizontal gradient does not vanish, that is, it is the complement of the set 
$$
\mathcal{C}(S)=\{x\in\frakH_\mH\,|\,|\nabla_0 u(x)| = 0\}.
$$
For the remainder of this section, we let the indices $i$ and $j$ range over $\{1, 2, 3, 4\}$ and $\{1, 2, 3\}$ respectively.
Recall the $L$-Riemannian metric $g_L$ on $\frakH_\mH$:
$$g_L=\sum_i dx_i^2+\sum_j L_j^2\theta_j^2.$$
Let $a_i=\xi_i u$ and $b_j=T_j' u$ denote the derivatives of $u$ with respect to the horizontal vector fields $\{\xi_i\}$ and the vertical vector fields $\{T_j'\}$.
We define the norms of the horizontal and $L$-Riemannian gradients, respectively, as
$$A=|\nabla_0 u|=\sqrt{\sum_i a_i^2} \quad \text{and} \quad A_L=\sqrt{\sum_i a_i^2+\sum_j b_j^2}.$$
The $L$-Riemannian unit normal vector field to the hypersurface $S$ is given by
$${\bf n}_L=\sum_{i}a_i^L\xi_i+\sum_{j} b_j^LT_j',$$
where $a_i^L=a_i/A_L$ and $b_j^L=b_j/A_L$.
When $L_j\to\infty,$ then $a_i^L\to a_i/A,$ while $b_j^L\to 0.$
In this way, the $L$-Riemannian normal ${\bf n}_L$ converges to the {\it horizontal normal vector field}
$${\bf n}_\calH=\sum_i \frac{\xi_i u}{|\nabla_0 u|}\xi_i,$$
which is the projection of ${\bf n}_L$ onto the horizontal subbundle $\mathcal{H}$.
\begin{defi}\label{def:hmc}
The {\it horizontal mean curvature} $\calH^0$ of $S$ is defined on the non-characteristic locus as
$$
\calH^0 =\sum_i \xi_i\left(\frac{\xi_i u}{|\nabla_0 u|}\right).
$$
\end{defi}
\begin{defi}
A $C^2$ regular hypersurface $S$ in $\frakH_\mH$ is called  {\it horizontally minimal} if its horizontal mean curvature $\calH_0$ vanishes ($\calH_0=0$) along its non-characteristic locus.
\end{defi}
\begin{eg}
The hyperplane $H_1 = \{x \in \frakH_\mH : U(x)=x_1 = 0\}$ is isomorphic to $\mathbb{R}^6$ as a submanifold. Direct computations show that $H_1$ is horizontally minimal. Note that the only characteristic point of $H_1$ is the origin.
\end{eg}
A rather non elementary paradigmatic example follows. Let $f$ be a smooth function of the radial horizontal variable $r=|q|.$ The following theorem provides an explicit formula for the horizontal mean curvature of a  kind of rotationally symmetric hypersurface in $\frakH_\mH,$ which may be significant interest in the study the isoperimetric problem within this setting. We let $\tau=|t|$.
\begin{thm}\label{prop:HMC}
The horizontal mean curvature $\calH^0$ of the hypersurface $u(x)=\tau-f(r)=0$
is given by
\begin{equation}\label{equ:form}
\calH^0=\frac{-4r(2f'+rf'')-3\frac{(f')^3}{r}}{(4r^2+(f')^2)^{{3}/{2}}}+\frac{8r^2}{f\sqrt{4r^2+(f')^2}},
\end{equation}
where $f'$ and $f''$ denote the first and second derivatives of $f$ with respect to $r,$ respectively.
\end{thm}
\begin{proof}
We compute first
$$
\begin{pmatrix}
\xi_1 \\
\xi_2\\
\xi_3\\
\xi_4
\end{pmatrix}u=G
\begin{pmatrix}
x_1\\
x_2\\
x_3\\
x_4
\end{pmatrix},
$$
where $G$ is the anti-diagonal matrix
$$
G=
\begin{pmatrix}
   -\frac{f'}{r} &  2\frac{t_1}{\tau} &  2\frac{t_2}{\tau} &  2\frac{t_3}{\tau} \\
   -2\frac{t_1}{\tau} &    -\frac{f'}{r} & 2\frac{t_3}{\tau} & -2\frac{t_2}{\tau}\\
   -2\frac{t_2}{\tau} & -2\frac{t_3}{\tau} &    -\frac{f'}{r} & 2\frac{t_1}{\tau}\\
   -2\frac{t_3}{\tau} & 2\frac{t_2}{\tau} & -2\frac{t_1}{\tau} &   -\frac{f'}{r}
\end{pmatrix}.
$$
Therefore,  the  norm of the horizontal gradient is  $|\nabla_0 u|=\sqrt{4r^2+(f')^2}.$
We next compute 
$$
\sum_i \xi_i (\xi_i u)=-\frac{f''r+3f'}{r}+\frac{8r^2}{\tau},
$$
which yields 
$$
\sum_i (\xi_i u)\,\xi_i(|\nabla_0 u|)=-\frac{4r+f'f''}{|\nabla_0 u|}f'.
$$
Therefore,
$$
\begin{aligned}
\calH^0&=\sum_i \xi_i\left(\frac{\xi_i u}{|\nabla_0 u|}\right)\\
&=\sum_i \frac{\xi_i(\xi_i u)}{|\nabla_0u|}-\sum_{i=1}^4 \frac{(\xi_i u)\xi_i(|\nabla_0u|)}{|\nabla_0u|^2}\\
&=\frac{1}{|\nabla_0u|}\left(-\frac{f''r+3f'}{r}+\frac{8r^2}{\tau}\right)+\frac{4r+f'f''}{|\nabla_0u|^3}f'\\
&=\frac{-4r(2f'+rf'')-3\frac{(f')^3}{r}}{(4r^2+(f')^2)^{{3}/{2}}}+\frac{8r^2}{f\sqrt{4r^2+(f')^2}}.
\end{aligned}
$$
\end{proof}
\begin{cor}\label{cor-0}
A hypersurface $u(x)=\tau-f(r)=0$ is horizontally minimal if
$$
-4r^2f(2f'+rf'')-3f(f')^3+8r^3(4r^2+(f')^2)=0.
$$
\end{cor}
\begin{eg}
The hypersurface $H_2=\left\{(q, t)\in\frakH_\mH: \tau=\sqrt{\frac{4}{3}}r^2\right\}$ is a horizontally minimal surface in the quaternionic Heisenberg group $\frakH_\mH,$  as this follows directly by applying formula (\ref{equ:form}) (or Corollary \ref{cor-0}). Note that its characteristic locus comprises only the origin.
\end{eg}
\begin{eg}
The Euclidean sphere is the hypersurface  $\{(q, t)\in\frakH_\mH: r^2+\tau^2=R^2\}.$
We set $f(r)=\sqrt{R^2-r^2}.$ Then applying formula (\ref{equ:form}) we find that
$$
\calH^0=\frac{3(4+R^2)+8r^2(4\tau^2+3)}{8r(\tau^2+1)^{3/2}}.
$$
\end{eg}
\begin{eg}
The Korányi sphere is the set $\{(q, t)\in\frakH_\mH: r^4+\tau^2=R^4\}.$ Then
 $\tau=f(r)=\sqrt{R^4-r^4}$ and it follows that
the horizontal mean curvature is
$$
\calH^0=\frac{9r}{R^2}.
$$
\end{eg}
\begin{eg}
The CC-sphere $S_R=\{(q, t)\in\frakH_\mH: d_{cc}(O, (q, t))=R \}.$ 
It follows from (\ref{equ:sphere}) that we have:
\[
r=\frac{2}{c}\sin\left(\frac{cR}{2}\right), \quad \tau=\frac{4(cR-\sin(cR))}{c^2},
\]
where $c\in \R.$ Set $r=f_1(c),$ and $\tau=f_2(c).$ Then the horizontal mean curvature can be determined using formula (\ref{equ:form}) by setting $f(r)=f_2\circ f_1^{-1}(r).$ 

The resulting algebraic expression is a rational function in $r, \tau$, and their partial derivatives with respect to the horizontal vector fields. Due to the inherent complexity of the quaternionic Heisenberg group structure, the full expression spans is exceedingly lengthy and is omitted here for brevity.
\end{eg}

\end{document}